\documentclass[12pt,twoside]{amsart}
\usepackage{amssymb}
\usepackage{verbatim}
\usepackage{amsmath}
\usepackage{bm}
\usepackage{a4wide}
\usepackage[latin1]{inputenc}
\usepackage[T1]{fontenc}
\usepackage{times}
\usepackage{amssymb,latexsym}
\usepackage{enumerate}
\newcommand{\op} {\overline{\partial}}
\newcommand{\dbar}{\ensuremath{\overline\partial}}

\makeatletter
\newcommand{\sumprime}{\if@display\sideset{}{'}\sum%
            \else\sum'\fi}
\makeatother

\begin{document}

\numberwithin{equation}{section}

\newtheorem{theorem}{Theorem}[section]
\newtheorem{proposition}[theorem]{Proposition}
\newtheorem{conjecture}[theorem]{Conjecture}
\def\theconjecture{\unskip}
\newtheorem{corollary}[theorem]{Corollary}
\newtheorem{lemma}[theorem]{Lemma}
\newtheorem{observation}[theorem]{Observation}
\newtheorem{definition}{Definition}
\numberwithin{definition}{section} 
\newtheorem{remark}{Remark}
\def\theremark{\unskip}
\newtheorem{kl}{Key Lemma}
\def\thekl{\unskip}
\newtheorem{question}{Question}
\def\thequestion{\unskip}
\newtheorem{example}{Example}
\def\theexample{\unskip}
\newtheorem{problem}{Problem}

\address{DEPARTMENT OF MATHEMATICAL SCIENCES, NORWEGIAN UNIVERSITY OF SCIENCE AND TECHNOLOGY, NO-7491 TRONDHEIM, NORWAY}
\email{xu.wang@ntnu.no}

\title[Alexandrov-Fenchel inequality]{A remark on the Alexandrov-Fenchel inequality}

 \author{Xu Wang}
\date{\today}

\begin{abstract} In this article, we give a complex-geometric proof of the Alexandrov-Fenchel inequality without using toric compactifications. The idea is to use the Legendre transform and develop the Brascamp-Lieb proof of the Pr\'ekopa theorem. New ingredients in our proof include an integration of Timorin's mixed Hodge-Riemann bilinear relation and a mixed norm version of H\"ormander's $L^2$-estimate, which also implies a non-compact version of the Khovanski\u{i}-Teissier inequality. 

\bigskip


\smallskip

\end{abstract}
\maketitle

\section{Introduction}

The classical Brunn-Minkowski inequality is an inequality on the volumes of convex bodies in $\mathbb R^n$. It plays an important role in many branches of mathematics, to quote from Gardner's survey article \cite{Gardner}: "In a sea of mathematics, the Brunn-Minkowski inequality appears like an octopus, tentacles reaching far and wide...". A far reaching generalization of it is the Alexandrov-Fenchel inequality, which has many different proofs (see section 20.3 in \cite{BZ}).  In 1936, Alexandrov found a combinatorial proof and an analytic proof. The later is a generalization of Hilbert's 1910 proof ("Minkowskis Theorie von Volumen und Oberfl\"ache") of the Brunn-Minkowski inequality. A simple algebraic proof (see \cite{KK12} and \cite{KK12-1}) based on the Bernstein-Kushnirenko theorem and the intersection theory on quasi-projective variety was given by Kaveh and Khovanski\u{i} around 2008. For other interesting proofs and related results, see \cite{GMTZ}, \cite{MR13}, \cite{DN06} and \cite{Cattani08}, to cite only a few. The Brunn-Minkowski inequality also has a functional version, i.e. the Pr\'ekopa theorem \cite{Prekopa73} for convex functions, which was found by Pr\'ekopa in 1973. In 1976 \cite{BL76},  Brascamp and Lieb gave another proof of the Pr\'ekopa theorem, the main idea is to use the Brascamp-Lieb lemma (see Lemma \ref{le:BLf-new}) to reduce the Pr\'ekopa theorem to a weighted $L^2$-estimate of H\"ormander type \cite{Hormander65} (so called the Brascamp-Lieb inequality)
for the minimal solution $u$ of 
$$
du=v.
$$
In 1998, by a magic way of using  H\"ormander's $\dbar$-$L^2$ estimate \cite{Hormander65}, Berndtsson \cite{B98} proved a complex version of the Pr\'ekopa theorem for plurisubharmonic functions. In 2005, inspired by \cite{BBN}, Cordero-Erausquin \cite{Co} discovered the relation between Berndtsson's work and the Brascamp-Lieb proof. 
Shortly after that, a very general and useful theory (so called the complex Brunn-Minkowski theory) \cite{Bern09, Bern06} behind the Brascamp-Lieb proof and Maitani-Yamaguchi's result \cite{MY04} was established by Berndtsson. The main result in that theory is a deep and beautiful curvature formula for a certain direct image bundle, which has found many highly non-trivial applications in K\"ahler geometry and algebraic geometry,  see \cite{Bern09, BPW,Bern-CBM, Bern-notes, Bo-notes} and references therein. Inspired by \cite{Wang-k} and Berndtsson's theory, in this paper we obtain a new complex-geometric proof of the Alexandrov-Fenchel inequality. The main idea is that the Brascamp-Lieb lemma (see Lemma \ref{le:BLf-new}) reduces the Alexandrov-Fenchel inequality to an  $L^2$-estimate $||u|| \leq ||\theta||$ on $\mathbb R^n \times (\mathbb R^n/\mathbb Z^n)$ for the minimal solution of
$$
du=(d^c)^*\theta, \ d^c:=i\dbar-i\partial,
$$
with respect to Timorin's mixed norm (see \cite{Timorin98} and \cite{Wang17}). The main advantage of this approach is that  \emph{we can prove the $L^2$-estimate $||u|| \leq ||\theta||$ directly, without using the compactification theory.} In fact, by H\"ormander's $L^2$-theory \cite{Hormander66, Demailly82},  it is enough to construct a special complete K\"ahler metric on $\mathbb R^n \times (\mathbb R^n/\mathbb Z^n)$ (Lemma \ref{le:complete}). Another advantage is  that the $L^2$-estimate $||u|| \leq ||\theta||$ is true on a large class of non-compact manifolds, not only on $\mathbb R^n \times (\mathbb R^n/\mathbb Z^n)$. In \cite{Gromov} (p 21), Gromov suggested to study non-compact generalizations of the Khovanski\u{i}-Teissier inequality. Our approach generalizes the Khovanski\u{i}-Teissier inequality
 to the following:
 
\begin{theorem}\label{th:1}  Let $(X, \hat \omega)$ be an $n$-dimensional complete K\"ahler manifold with finite volume. Let $\alpha_{1}, \cdots, \alpha_n$ be smooth $d$-closed semi-positive $(1,1)$-forms  such that $\alpha_{j}\leq \hat \omega$ on $X$ for every $1\leq j\leq n$. Assume that $n\geq 2$. Put
$$
T:=\alpha_{3} \wedge \cdots \wedge \alpha_n,\ T:=1, \ \text{if}\ n=2.
$$
Then
$$
\left(\int_X \alpha_{1}\wedge \alpha_2\wedge T \right)^2 \geq \left(\int_{X} \alpha_{1}^2\wedge T\right)\left(\int_X \alpha_{2}^2\wedge T\right).
$$
\end{theorem}

\textbf{Remark}: The above theorem can be seen as a special case of our main result (Theorem \ref{th:main}). Recall that a Hermitian manifold $(X, \hat \omega)$ is said to be \emph{complete} if there exists a smooth function, say 
$$
\rho: X\to [0, \infty),
$$
such that $\rho^{-1}([0, c])$ is compact for every $c>0$ and
$$
|d\rho|_{\hat \omega} (x) \leq 1, \ \forall \ x\in X.
$$
In order to deduce the classical Alexandrov-Fenchel inequality from Theorem \ref{th:1}, we construct a special complete K\"ahler metric on $\mathbb R^n \times (\mathbb R^n/\mathbb Z^n)$ in Lemma \ref{le:complete}. The whole paper is organized as follows.

\tableofcontents

\emph{Acknowledgement}: The author would like to thank Professor Bo Berndtsson for many inspiring discussions on the Alexandrov-Fenchel inequality and related topics. Thanks are also given to Professor Bo-Yong Chen and Professor Qing-Chun Ji for their constant support and encouragement. Last but not least, thanks are due to the referee for many helpful suggestions. 
The author was partially supported by the Knut and Alice Wallenberg Foundation and the Onsager fellowship.

\section{Preliminaries}

\subsection{Basic notions in convex geometry}

\begin{enumerate}
\item A set $\Omega$ in $\mathbb R^n$ is said to be \emph{convex} if the line segment between any two points in $\Omega$ lies in $\Omega$.

\item We call a \emph{compact convex set, say $A$, with non-empty interior}, say $A^\circ$, in $\mathbb R^n$ a \emph{convex body}. 
\end{enumerate}

Let $A_0$, $A_1$ be two convex bodies in $\mathbb R^n$. We call 
\begin{equation*}
A_0+A_1:=\{a_0+a_1 \in \mathbb R^n: a_0\in A_0, \ a_1\in A_1\},
\end{equation*}
the \emph{Minkowski sum} of $A_0$ and $A_1$. The Brunn-Minkowski theorem (see \cite{Gardner} for a nice survey) reads as follows:

\begin{theorem}[Brunn-Minkowski inequality] $|A_0+A_1|^{1/n} \geq |A_0|^{1/n} +|A_1|^{1/n}$, where the absolute value of a convex body means its volume (Lebesgue measure).  
\end{theorem} 

\textbf{Remark}: The Brunn-Minkowski inequality is also true for compact \emph{non-convex} sets with non-empty interior, see \cite{Ly}.

\medskip

We will also need the following notion in convex geometry. 

\begin{definition}[Legendre transform] Let $A$ be a convex body.  Let $\psi$ be a smooth real-valued function on $A^\circ$. $\psi$ is said to be strictly convex if the Hessian matrix $(\psi_{jk})$ is positive definite at every point in $A^\circ$. We call 
$$
\psi^*(y):=\sup_{x\in A^{\circ}} x\cdot y-\psi(x), \ x\cdot y:=\sum_{j=1}^n x^j y^j,
$$
the Legendre transform of $\psi$ (with respect to $A^\circ$). 
\end{definition} 

\begin{proposition}\label{pr:partial-convex-convex}  
Let $\psi$ be a smooth strictly convex function that tends to infinity at the boundary of a convex body $A$. Then its Legendre transform  $\psi^*$ is also smooth, strictly convex,  moreover  the gradient map of $\psi^*$
 \begin{equation}
\nabla \psi^*: y\mapsto x=\nabla \psi^*(y):=(\partial\psi^*/\partial x^1, \cdots, \partial\psi^*/\partial x^n),
\end{equation}
defines a diffeomorphism from $\mathbb R^n$ onto $A^\circ$. 
\end{proposition}

\begin{proof}

It is enough to prove that the  gradient map of $\psi$ defines a diffeomorphism from $A^\circ$ to $\mathbb R^n$, $\psi^*$ is smooth and $\nabla \psi^*$ is the inverse of $\nabla \psi$.

\medskip

\emph{Step 1: $\nabla\psi$ is a diffeomorphism from $A^\circ$ to $\mathbb R^n$.} Since $\psi$ is smooth and strictly convex, we know that $\nabla\psi$ is a local diffeomorphism.

\medskip

1.  $\nabla\psi$ is injective: assume that $\nabla\psi(x_1)=\nabla\psi(x_2)=y_0$, consider 
\begin{equation}
\psi^{y_0}(x):= \psi(x)- y_0\cdot x,
\end{equation}
we know that $\psi^{y_0}$ is smooth, strictly convex and 
\begin{equation}
\nabla\psi^{y_0}(x_1)=\nabla\psi^{y_0}(x_2)=0.
\end{equation}
Consider the restriction, say $g$, of $\psi^{y_0}$ to the line determined by $x_1$ and $x_2$, then $g$ is convex with critical points $x_1$ and $x_2$. Thus $g$ is a constant on the line segment from $x_1$ to $x_2$, moreover, strict convexity of $g$ implies $x_1=x_2$. Thus $\nabla\psi$ is injective.

\medskip

2. $\nabla\psi(A^0)=\mathbb R^n$: fix $y\in \mathbb R^n$, since $\psi^y$ tends to infinity at the boundary of $A$, strict convexity of $\psi$ implies that $\psi^y$ has a unique minimum point, say $x\in A^\circ$. Thus
$$
0=\nabla\psi^y(x)=\nabla\psi(x)-y.
$$ 

\medskip

\emph{Step 2: $\psi^*$ is smooth.} Notice that
\begin{equation}\label{eq:psi-star}
\psi^*(\nabla\psi(x))=\nabla\psi(x)\cdot x -\psi(x).
\end{equation}
Thus $\psi^*\circ \nabla\psi$ is a smooth, which implies that $\psi^*$ is smooth on $\mathbb R^n$. 

\medskip

\emph{Step 3: $\nabla \psi^*$ is the inverse of $\nabla \psi$.} Apply the differential to \eqref{eq:psi-star}, we get that
\begin{equation}\label{eq:psi-psi-star}
(\nabla \psi^* \circ \nabla\psi (x)) \cdot (\psi_{jk})=x\cdot (\psi_{jk}), \ \forall \ x\in A^\circ.
\end{equation} 
Since $(\psi_{jk})$ is an invertible matrix function, the above formula gives $\nabla \psi^* \circ \nabla\psi =Id$. 
\end{proof}

\textbf{Remark}: Put $\phi=\psi^*$. We know from the above proposition that $\nabla\phi$ is a diffeomorphism from $\mathbb R^n$ onto the interior of $A$, thus
\begin{equation}\label{eq:start}
|A|=\int_{A} dy=\int_{\mathbb R^n} MA(\phi)\, dx, \ dx:=dx^1\wedge \cdots \wedge dx^n, \  dy:=dy^1\wedge \cdots \wedge dy^n.
\end{equation} 
where $MA(\phi):=\det(\phi_{jk})$ denotes the determinant of the Hessian of $\phi$. In case $A$ is the convex hull of a finite set, say $\{p_j\}_{1\leq j\leq N} \subset \mathbb R^n$, one may choose
$$
\phi(x)=\log \left(\sum_{j=1}^N e^{p_j\cdot x}\right).
$$
For more results on convex function of the above type, see \cite{Nystrom15} and \cite{Gromov}, see also \cite{B-B} and \cite{CK} for the canonical choice of such $\phi$. 

\medskip

The following proposition is a generalization of \eqref{eq:start}.

\begin{proposition}\label{pr:phi+phi}  Let $\phi_1, \cdots, \phi_N$ be smooth strictly convex functions such that each $\nabla\phi_j$ is a diffeomorphism from $\mathbb R^n$ onto the interior of a convex body $A_j$. Then we have
\begin{equation}\label{eq:phi-phi}
|t_1 A_1+\cdots +t_N A_N|=\int_{\mathbb R^n} MA(t_1\phi_1+\cdots +t_N\phi_N) \, dx, \ t_j >0, \ \forall \ 1\leq j\leq N.
\end{equation}
\end{proposition}

\begin{proof} By induction on $N$, it suffices to show that
\begin{equation}\label{eq:0123}
\nabla(\phi_1+\phi_2)(\mathbb R^n)=A^\circ_1+A^\circ_2,
\end{equation}
where $A^\circ$ denotes the interior of $A$. Obviously we have $\nabla(\phi_1+\phi_2)(\mathbb R^n)\subset A^\circ_1+A^\circ_2$.  Thus it is enough to show that for every $y_1\in A_1^\circ$ and every $y_2\in A_2^\circ$, there exists $x_0\in \mathbb R^n$ such that $\nabla(\phi_1+\phi_2)(x_0)=y_1+y_2$. Consider $\phi_j^{y_j}$ instead of $\phi_j$, one may assume that $y_1=y_2=0$. Choose $x_1$ and $x_2$ such that 
\begin{equation}
\nabla\phi_1(x_1)=\nabla\phi_2(x_2)=0. 
\end{equation}
Since $\phi_j$ is convex, we know that each $x_j$ is the minimum point of $\phi_j$. Thus strict convexity of $\phi_j$ implies that
\begin{equation}
\phi_j(x) \to \infty, \ \text{as} \ |x| \to\infty,
\end{equation}
i.e. each $\phi_j$ is proper. Thus $\phi_1+\phi_2$ is also proper. Hence there exists a unique minimum point, say $x_0$, of $\phi_1+\phi_2$. Thus $\nabla(\phi_1+\phi_2)(x_0)=0$. The proof is complete.
\end{proof}

\textbf{Remark}: The above proposition implies that
\begin{equation}
p(t):=|t_1 A_1+\cdots +t_n A_n|,
\end{equation}
is a polynomial of degree $n$.  We call the coefficient of $t_1\cdots t_n$ in the polynomial $p(t)$, i.e.
$$
V(A_1, \cdots, A_n):=\frac{\partial^n |t_1 A_1+\cdots +t_n A_n|}{\partial t_1 \cdots \partial t_n}, 
$$
 the \emph{mixed volume} of $A_1, \cdots, A_n$. 
 
\subsection{Alexandrov-Fenchel inequality} 

\begin{theorem}[Alexandrov-Fenchel inequality] Let $A_1, \cdots, A_n$ be convex bodies in $\mathbb R^n$. Assume that $n\geq 2$. Then $$V(A_1, \cdots, A_n)^2\geq V(A_1,A_1, A_3,\cdots, A_n) V(A_2,A_2, A_3,\cdots, A_n).$$  
\end{theorem} 

The following lemma can be used to find equivalent forms of the Alexandrov-Fenchel inequality.

\begin{lemma}\label{le:one-h} Let $f$ be a positive smooth function on an open convex cone, say $\mathcal K$,  in $\mathbb R^N$. Assume that $f$ is $1$-homogeneous, i.e. 
\begin{equation*}
f(tx)\equiv t f(x), \ \forall \ t>0, \ x\in \mathcal K.
\end{equation*}
Then the following statements are equivalent: 

$A1$: $f(x+y)\geq  f(x)+f(y), \ \forall \ x, y\in\mathcal K$;

$A2$:  $-f$ is convex; 

$A3$:  $-\log f$ is convex;

$A4$:  For every $x', y' \in \mathcal K$, $ t\mapsto -\log f(tx'+(1-t)y')$ is convex on $(0,1)$. 

\end{lemma}

\begin{proof} Since $f$ is $1$-homogeneous, $A1$ implies
\begin{equation}
f(tx+(1-t)y) \geq tf(x)+(1-t)f(y).
\end{equation}
Thus $A1 \Rightarrow A2$. Since 
\begin{equation}
(-\log f)_{\xi\xi}=\frac{-f_{\xi\xi}}{f}+\frac{(f_\xi)^2}{f^2}, \  f_{\xi}=\sum\xi^j f_{x_j},
\end{equation}
we know $A2\Rightarrow A3$. Since $A3\Rightarrow A4$ is trivial, it is enough to show $A4\Rightarrow A1$: notice that $A4$ implies
\begin{equation}
f(tx'+(1-t)y')\geq f(x')^t f(y')^{1-t}.
\end{equation}
Take
\begin{equation}
x'=\frac{x}{f(x)}, \ y'=\frac{y}{f(y)}, \ t=\frac{f(x)}{f(x)+f(y)}, 
\end{equation}
we get $A1$. The proof is complete.
\end{proof}

Apply the above lemma to the following function
\begin{equation}
f(x)=V(A_x, A_x, A_3, \cdots, A_n)^{1/2}, \  A_x:= x_1A_1+x_2 A_2, 
\end{equation}
on $\mathcal K:=\mathbb R_+^2$. Notice that the square of  
\begin{equation}
f(x+y)\geq f(x)+f(y), 
\end{equation}
is equivalent to
$$
V(A_x, A_y, A_3, \cdots, A_n)^{2}\geq V(A_x, A_x, A_3, \cdots, A_n) V(A_y, A_y, A_3, \cdots, A_n).
$$
By the above lemma, we have

\begin{proposition}\label{pr:AF-log} The Alexandrov-Fenchel  inequality is equivalent to the convexity of
$$
t\mapsto -\log V(A_t, A_t, A_3, \cdots, A_n), \ A_t:=t A_1+(1-t)A_2,
$$
on $(0,1)$. 
\end{proposition}

A generalized form of the Alexandrov-Fenchel inequality is also true. 

\begin{theorem}\label{th:AF-m} Let $A_1, A_2, A_{m+1}, \cdots, A_n$, $2\leq m\leq n$, be convex bodies in $\mathbb R^n$. Then the following function is convex on $(0,1)$
$$
t\mapsto -\log V(\underbrace{ A_t, \cdots, A_t}_m, A_{m+1}, \cdots, A_n), \ A_t:=t A_1+(1-t)A_2.
$$ 
\end{theorem}

The above theorem is in fact equivalent to the Alexandrov-Fenchel inequality (see Theorem 7.4.5 in \cite{Schneider}).

\subsection{Khovanski\u{i}-Teissier inequality}

We will use the following complex geometry interpretation of the volume function in Proposition \ref{pr:phi+phi}.

\begin{lemma}\label{le:28-new}  Let $\phi_1, \cdots, \phi_N$ be smooth strictly convex functions such that each $\nabla\phi_j$ is a diffeomorphism from $\mathbb R^n$ onto the interior of a convex body $A_j$. Let us look at 
$$
\phi:=\sum_{j=1}^N t_j \phi_j,
$$
as a function on 
$$
\mathbb R^n \times {\mathbb T}^n=\mathbb C^n/ i\mathbb Z^n, \ {\mathbb T}:=\mathbb R/ \mathbb Z, \  i:=\sqrt{-1},
$$
i.e. $\phi(x+iy):=\sum_{j=1}^N t_j \phi_j(x)$. Then we have
$$
\int_{\mathbb R^n} MA(\phi) ~ dx=\int_{\mathbb R^n \times {\mathbb T}^n} \frac{(dd^c\phi)^n}{n!}, \ d^c:=i\dbar-i\partial.
$$
\end{lemma}

\begin{proof} Since
$$
dd^c\phi=2i\partial\dbar \phi =\frac{i}{2} \sum_{j,k=1}^n \phi_{jk} \,dz^j \wedge d\bar z^k, \ z^j:=x^j+iy^j, 
$$
where $\phi_{jk}:=\partial^2\phi/\partial x^j \partial x^k$, we have
$$
\frac{(dd^c\phi)^n}{n!}=\det(\phi_{jk})\, (dx^1\wedge dy^1) \wedge \cdots \wedge (dx^n\wedge dy^n),
$$
thus the lemma follows from the Fubini theorem and $\int_{{\mathbb T}^n} dy=1$.
\end{proof}

The above lemma implies

\begin{lemma}  Let $\phi_1, \cdots, \phi_n$ be smooth strictly convex functions such that each $\nabla\phi_j$ is a diffeomorphism from $\mathbb R^n$ onto the interior of a convex body $A_j$. Then we have the following mixed volume formula
$$
V(A_1,\cdots, A_n)=\int_{\mathbb R^n\times {\mathbb T}^n} dd^c\phi_1 \wedge\cdots \wedge dd^c\phi_n. 
$$
\end{lemma}

\begin{proof} The previous lemma gives
$$
|\sum_{j=1}^n t_j A_j|=\int_{\mathbb R^n \times {\mathbb T}^n} \frac{(dd^c\phi)^n}{n!}, \ t_j >0, \ \forall \ 1\leq j\leq n.
$$
Notice that
$$
\frac{(dd^c\phi)^n}{n!}=\sum_{\alpha_1+\cdots+\alpha_n=n}  \frac{t_1^{\alpha_1} \cdots t_n^{\alpha_n}}{\alpha_1 ! \cdots \alpha_n !}  \  (dd^c\phi_1)^{\alpha_1} \wedge\cdots \wedge (dd^c\phi_n)^{\alpha_n}, 
$$
and each term $(dd^c\phi_1)^{\alpha_1} \wedge\cdots \wedge (dd^c\phi_n)^{\alpha_n}$ is a positive $(n,n)$-form, thus
$$
|\sum_{j=1}^n t_j A_j|<\infty \Rightarrow  \int_{\mathbb R^n \times {\mathbb T}^n} (dd^c\phi_1)^{\alpha_1} \wedge\cdots \wedge (dd^c\phi_n)^{\alpha_n} <\infty.
$$  
Now we have
$$
|\sum_{j=1}^n t_j A_j|=\sum_{\alpha_1+\cdots+\alpha_n=n}   \frac{t_1^{\alpha_1} \cdots t_n^{\alpha_n}}{\alpha_1 ! \cdots \alpha_n !}  \int_{\mathbb R^n \times {\mathbb T}^n} (dd^c\phi_1)^{\alpha_1} \wedge\cdots \wedge (dd^c\phi_n)^{\alpha_n},
$$
and the lemma follows.
\end{proof}

By the above lemma, we know that Theorem \ref{th:AF-m} is equivalent to the following:

\begin{theorem}\label{th:AF-m-c}  Let $\phi_1, \phi_2,\, \phi_{m+1}, \cdots, \phi_n$, $2\leq m \leq n$, be smooth strictly convex functions such that each $\nabla\phi_j$ is a diffeomorphism from $\mathbb R^n$ onto the interior of a convex body $A_j$. Then the following function is convex on $(0,1)$
$$
t\mapsto -\log  \int_{\mathbb R^n\times {\mathbb T}^n}  \frac{\omega^m}{m!} \wedge T, 
$$
where 
$$
\omega:=t dd^c\phi_1+(1-t) dd^c\phi_2, \qquad T:=dd^c\phi_{m+1}\wedge \cdots \wedge dd^c\phi_n.
$$
\end{theorem} 

Let us recall the following Khovanski\u{i}-Teissier theorem.

\begin{theorem}[Khovanski\u{i}-Teissier inequality] Let $\omega_1,\cdots, \omega_n$ be K\"ahler forms on a compact K\"ahler manifold $X$. Assume that $n\geq 2$. Put
$$
T:=\omega_{3} \wedge \cdots \wedge \omega_n,\qquad T:=1, \ \text{if}\ n=2.
$$
Then
$$
\left(\int_X \omega_{1}\wedge \omega_2\wedge T \right)^2 \geq \left(\int_{X} \omega_{1}^2\wedge T\right)\left(\int_X \omega_{2}^2\wedge T\right).
$$
\end{theorem}

By Lemma \ref{le:one-h}, we know that the Khovanski\u{i}-Teissier inequality is equivalent to the ($m=2$ case) convexity of
$$
t\mapsto -\log  \int_{X}  \frac{\omega^m}{m!} \wedge T,  \qquad \omega:=t\omega_1 +(1-t)\omega_2, \ T:=\omega_{m+1}\wedge \cdots\wedge \omega_n.
$$
Thus Theorem \ref{th:AF-m-c} can be seen as a Khovanski\u{i}-Teissier inequality for $\mathbb R^n\times {\mathbb T}^n$.

\medskip

\textbf{Remark}: The above equivalent description of the Khovanski\u{i}-Teissier inequality was first used by Graham in his proof of the convexity of the interpolating function, see \cite{Graham}. There are also other descriptions of the Khovanski\u{i}-Teissier inequality.  A very nice intersection theory description of its algebraic version can be found in  \cite{Khovanskii88} and \cite{KK12}. In the Hodge theory description, the Khovanski\u{i}-Teissier inequality is a direct application of the \emph{mixed generalization of the classical Hodge-Riemann bilinear relation} (MHRR) for $(1,1)$-forms.  MHRR for general $(p,q)$-forms on a compact K\"ahler manifold was first proved by Dinh-Nguy\^en in \cite{DN06} based on Timorin's result \cite{Timorin98} for the torus case, see also \cite{Cattani08} for another approach that applies to general polarized Hodge-Lefschetz modules. 

\section{Main theorem} 

\begin{theorem}\label{th:main} Let $(X, \hat \omega)$ be an $n$-dimensional complete K\"ahler manifold with finite volume. Let $\alpha_{1}, \alpha_2, \alpha_m, \cdots, \alpha_n$, $2\leq m\leq n$, be smooth $d$-closed semi-positive $(1,1)$-forms  such that each $\alpha_{j}\leq \hat \omega$ on $X$. Then the following function is convex on $(0,1)$ 
$$
t\mapsto -\log\int_X \frac{\omega^m}{m!} \wedge T, \qquad \omega:=t\alpha_1+(1-t)\alpha_2,
$$
where $T:=\alpha_{m+1} \wedge \cdots \wedge \alpha_n,\ T:=1, \ \text{if}\ n=m$.
\end{theorem}

By Lemma \ref{le:one-h}, in case $m=2$, our main theorem is equivalent to Theorem \ref{th:1}, which is a non-compact generalization of the Khovanski\u{i}-Teissier inequality.

\medskip

\textbf{About the proof of the main theorem}. Put
$$
f(t)= -\log\int_X \frac{\omega^m}{m!} \wedge T.
$$
Consider $\alpha_j+\epsilon \hat \omega $ instead of $\alpha_j$ and denote by $f^{\epsilon}$ the associated function. Then we have
$$
f=\lim_{\epsilon \to 0}  f^{\epsilon}.
$$
Thus it suffices to show that each $f^{\epsilon}$ is convex on $(0,1)$, i.e. one may assume that 
\begin{equation}\label{eq:equ-norm}
\frac{\hat \omega}{C} \leq \alpha_j \leq C \hat \omega,  
\end{equation}
for every $j$ in Theorem \ref{th:main}, where $C$ is a fixed positive constant. Then Theorem \ref{th:main} follows from the following three lemmas.

\begin{lemma}\label{le:1-new} Assume that \eqref{eq:equ-norm} is true. Define $G$ on $X$ such that 
$$
\frac{d}{dt}\left (\frac{\omega^m}{m!} \wedge T\right) =-G\, \frac{\omega^m}{m!} \wedge T.
$$
Then
$$
f_{tt}:=\frac{d^2 f} {dt^2} =\int_{X} \left(G_t- (G-E_\mu(G))^2 \right) \, d\mu,
$$
where
$$
d\mu:= \frac{ \frac{\omega^{m}}{m!} \wedge T}{ \int_{X} \frac{\omega^{m}}{m!} \wedge T },  \ E_\mu(G):= \int_{X} G \, d\mu.
$$
\end{lemma}

\begin{lemma}\label{le:2-new} Assume that \eqref{eq:equ-norm} is true. Then
\begin{equation}\label{eq:Gt-new}
\int_{X} G_t \, d\mu= e^f ||\theta||^2_{T, \omega}\, , \qquad \theta:=\frac{d}{dt} \omega= \alpha_1-\alpha_2,
\end{equation}
and
\begin{equation}\label{eq:G-new}
\int_{X} (G-E_\mu(G))^2 \, d\mu= e^f ||G-E_\mu(G)||^2_{T, \omega},
\end{equation}
where $||\cdot||_{T,\omega}$ denotes the $T$-Hodge theory norm (see Definition \ref{de:T-norm}). Moreover, 
\begin{equation}\label{eq:Gtheta-new}
T\wedge G=-\Lambda(T\wedge \theta),
\end{equation}
where $\Lambda$ denotes the adjoint of $\omega \wedge\cdot$ in $T$-Hodge theory.
\end{lemma}

\begin{lemma}\label{le:3-new} Assume that \eqref{eq:equ-norm} is true. Then $T\wedge (E_\mu(G)-G)$ is the $L^2$-minimal solution of 
$$
d(\cdot)=(d^c)^*(T\wedge \theta),
$$
with respect to the $T$-Hodge theory norm and 
$$
||G-E_\mu(G)||_{T, \omega}\leq ||\theta||_{T, \omega}.
$$
\end{lemma}

\section{Brascamp-Lieb lemma}

We shall use the Brascamp-Lieb lemma to prove Lemma \ref{le:1-new}. 

\subsection{Brascamp-Lieb proof of the Pr\'ekopa theorem}

The following Pr\'ekopa theorem was found by Pr\'ekopa around 1973. 

\begin{theorem}[Pr\'ekopa's theorem \cite{Prekopa73}] Let $\phi$ be a smooth, strictly convex function of $(t,x)$ in $\mathbb R^{n+1}$.  Then
\begin{equation}\label{eq:BM2}
t\mapsto -\log \int_{A} e^{-\phi(t, x)}\, d\lambda(x),
\end{equation}
is strictly convex on $\mathbb R$, where $A$ is a fixed convex body in $\mathbb R^n$ and $d\lambda(x)$ denotes the Lebesgue measure.
\end{theorem}

The Brascamp-Lieb proof in \cite{BL76} contains three steps.

\medskip

\emph{Step 1}:  The second order derivative of function \eqref{eq:BM2} can be written as
\begin{equation}\label{eq:BL}
\int_{A} \phi_{tt}- (\phi_t -E_\nu (\phi_t))^2 \,d\nu,
\end{equation}
where
\begin{equation}
d\nu:=\frac{e^{-\phi(t, x)}\, d\lambda(x)}{ \int_{A} e^{-\phi(t, x)}\, d\lambda(x) }, \ E_\nu(\phi_t):=\int_{A} \phi_t \, d\nu. 
\end{equation}

\medskip

\emph{Step 2}:  Prove the following Brascamp-Lieb inequality:
$$
\int_{\mathbb R^n} (\phi_t -E_\nu (\phi_t))^2\, d\nu \leq \int_{\mathbb R^n} \sum_{j,k=1}^n \phi_{tj}\phi^{jk}\phi_{tk} \, d\nu, 
$$
where $(\phi^{jk})$ denotes the inverse matrix of $(\phi_{jk})$.

\medskip

\emph{Step 3}: Use strict convexity of $\phi$ to prove $\phi_{tt} > \sum_{j,k=1}^n \phi_{tj}\phi^{jk}\phi_{tk}$.

\medskip

\textbf{Remark}: The first step follows from the following lemma (take  $dV=e^{-\phi} \,d\lambda$). Since $$\phi_t -E_\nu (\phi_t)$$ is the (weighted) $L^2$-minimal solution of $d(\cdot)=d(\phi_t)$, an H\"ormander type $L^2$-estimate gives step 2, see also \cite{BL76} for a direct proof. For step 3, let $D_{t,x}$ be the determinant of the full hessian matrix of $\phi$, let $D_{x}$ be the determinant of the hessian matrix of $\phi$ as a function of $x$, then 
$$
\frac{D_{t,x}}{D_x}= \phi_{tt} - \sum_{j,k=1}^n \phi_{tj}\phi^{jk}\phi_{tk}.
$$
Strict convexity of $\phi$ implies $D_{t,x} > 0$ and $D_x >0$. Thus Step 3 follows. 

\begin{lemma}[Brascamp-Lieb lemma]\label{le:BLf-new} Let $A$ be a relatively compact open set in a smooth manifold $X$. Let $\{dV(t)\}_{t\in\mathbb R}$ be a smooth family of smooth volume forms on $X$. Let us define $G$ such that
$$
\frac{d}{dt} dV(t) =-G(t,x) \,dV(t), \qquad  (t,x)\in \mathbb R\times X.
$$
Then
$$
\frac{d^2}{dt^2} \left( -\log \int_A dV(t)  \right) =\int_A \left( G_t- (G-E_\mu(G))^2 \right) \, d\mu,
$$
where
$$
d\mu:=\frac{dV}{\int_A dV}, \qquad  E_\mu(G):=\int_A G \,d\mu.
$$
\end{lemma}

\begin{proof}  Since $A$ is relatively compact, we have 
$$
\frac{d}{dt} \left( -\log \int_A dV(t)  \right) =\int_A G \,d\mu.
$$
Apply the differential again, we get
$$
\frac{d^2}{dt^2} \left( -\log \int_A dV(t)  \right) =\int_A G_t  \,d\mu+ G \frac{d}{dt} d\mu.
$$
A direct computation gives
$$
\frac{d}{dt} d\mu=-G \, d\mu+ E_\mu(G) \, d\mu,
$$
which implies $\int_A G \frac{d}{dt} d\mu= -\int_A (G-E_\mu(G))^2  \ d\mu$. Thus the lemma follows.
\end{proof}

\textbf{Remark}: In \cite{Bern09}, Berndtsson proved that the Brascamp-Lieb lemma is essentially a subbundle curvature formula associated to a certain direct image bundle. Our main theorem can also be proved along this line, see \cite{Wang17, Wang-k}. Other interesting formulas for the second order derivative of $-\log\int dV$ can be found in \cite{BBN}.

\subsection{Proof of Lemma \ref{le:1-new}}  Notice that the Brascamp-Lieb lemma gives Lemma \ref{le:1-new} if $X$ is compact. In case $X$ is non-compact we can not directly apply the Brascamp-Lieb lemma. In our case the main point is that 
$$
e^{-f}=\int_X  \frac{\omega^m}{m!} \wedge T,
$$
is a polynomial of degree $m$. The reason is that we can write
$$
\frac{\omega^m}{m!} \wedge T =\sum_{j=1}^m t^j \Omega_j.
$$ 
Then \eqref{eq:equ-norm} implies that each $\int_X \Omega_j$ is finite and
$$
e^{-f}=\sum_{j=1}^m \left(\int_X \Omega_j \right) t^j.
$$
Thus in our case, $\int_X$ commutes with $\frac{d}{dt}$  and the Brascamp-Lieb lemma applies.

\section{Timorin's $T$-Hodge theory}

We shall use Timorin's $T$-Hodge theory to prove Lemma \ref{le:2-new}. The motivation comes from the Brunn-Minkowski case, i.e. $T=1$ and $X=\mathbb R^n\times {\mathbb T}^n$ (recall $\mathbb T:=\mathbb R/\mathbb Z$).

\subsection{Brunn-Minkowski inequality}

By Lemma \ref{le:one-h}, we know that the Brunn-Minkowski inequality is equivalent to the convexity of 
$$
f : t \mapsto -\log |A_t|, \ A_t:=tA_1 +(1-t) A_2,
$$
on $(0,1)$. Let $\phi_1$ and $\phi_2$ be smooth strictly convex functions  that tend to infinity at the boundary of $A_1$ and $A_2$ respectively. Put
$$
\psi_1:=\phi_1^*, \ \psi_2:=\phi_2^*.
$$
Proposition \ref{pr:partial-convex-convex} gives
$$
\nabla \psi_1 (\mathbb R^n)=A^\circ_1, \ \nabla \psi_2 (\mathbb R^n)=A_2^\circ. 
$$
Thus by Proposition \ref{pr:phi+phi} we have
$$
|A_t|=\int_{\mathbb R^n} \det(\phi_{jk}) \,dx, \ \phi:=t\psi_1+(1-t)\psi_2.
$$
Apply the Brascamp-Lieb lemma to 
$$
dV=\det(\phi_{jk})\,dx,
$$
we get
\begin{equation}\label{eq:ftt}
f_{tt}=\int_{\mathbb R^n} G_t-(G- E_\mu (G))^2\,  d\mu,
\end{equation}
where
$$
d\mu:=\frac{\det(\phi_{jk}) \,d\lambda(x)}{ \int_{\mathbb R^n} \det(\phi_{jk})\, d\lambda(x) }, \ E_\mu(G):=\int_{\mathbb R^n} G \, d\mu. 
$$

\begin{lemma}\label{le:3.2} $G=-\sum_{j,k=1}^n \phi_{tjk} \phi^{jk}$.
\end{lemma}

\begin{proof} We use the fact that  if $M(t)$ is a smooth family of positive definite matrices then
$$
(\log \det M)_t={\rm Trace} (M^{-1}M_t).
$$
Consider $M=(\phi_{jk})$ then $G=-{\rm Trace} (M^{-1}M_t)$ and the lemma follows.
\end{proof}

\begin{lemma}\label{le:3.3} $G_t=\sum_{j,k,l,m=1}^n \phi_{tjk} \phi_{tlm}\phi^{jl}\phi^{km} $.
\end{lemma}

\begin{proof} If $M(t)$ is a smooth family of positive definite matrices then
$$
(M^{-1})_t=-M^{-1}M_t M^{-1}.
$$
Apply the above fact, we get 
$$
(\phi^{jk})_t= -\sum_{l,m=1}^n \phi_{tlm}\phi^{jl}\phi^{km}.
$$ 
Moreover, Lemma \ref{le:3.2} implies $G_t=-\sum_{j,k=1}^n  \phi_{tjk} (\phi^{jk})_t$, thus the lemma follows.
\end{proof}

By Lemma \ref{le:28-new}, we have
$$
f=-\log \int_{\mathbb R^n \times {\mathbb T}^n} \frac{(dd^c\phi)^n}{n!}.
$$
Consider $\omega=dd^c\phi$. The above two lemmas give
$$
G=-\Lambda\theta, \  G_t=|\theta|^2_{\omega},
$$
thus Lemma \ref{le:2-new} is true in case $T=1$ and $X=\mathbb R^n\times {\mathbb T}^n$.

\subsection{$T$-Hodge theory} In this subsection, we will introduce the $T$-Hodge theory behind the proof of Lemma \ref{le:2-new}. The $T$-Hodge theory is an integration of Timorin's work in \cite{Timorin98}, see the author's notes \cite{Wang17} for a systematic study of the $T$-Hodge theory.

\medskip

Denote by $V^{p,q}$ the space of smooth $(p,q)$-forms on an $n$-dimensional complex manifold $X$. Put
$$
V:=\oplus_{0\leq p,q \leq n} V^{p,q}, \ V^k:=\oplus_{ p+q = k} V^{p,q}.
$$

\begin{definition} Let 
$$
T=\alpha_{m+1} \wedge \cdots \wedge\alpha_{n},
$$
be a finite wedge product of smooth positive $(1,1)$-forms on $X$. We call the Hodge theory on $V_T:=\{T\wedge u: u\in V\}$ the $T$-Hodge theory.
\end{definition}

For bidegree reason,  we have
$$
V_T=\oplus_{0\leq p,q\leq m} V_T^{p,q},
$$ 
where $V_T^{p,q} $ denotes the space of forms  that can be written as $T\wedge u$, where $u$ is a smooth $(p,q)$-form on $X$. Fix a smooth positive $(1,1)$-form $\omega$ on $X$. The $L$ operator 
$$
L: T\wedge u\mapsto \omega\wedge T\wedge u, 
$$
is well defined and maps $V_T^{p,q}$ to $V_T^{p+1,q+1}$. 

\begin{theorem}[Timorin's mixed hard-Lefschetz theorem]\label{th: T-HL} Put $V_T^k=\oplus_{p+q=k} V_T^{p,q}$ then 
$$
L^{m-k}: T\wedge u\mapsto T\wedge u\wedge\omega^{m-k},\  0\leq k\leq m,
$$
defines an isomorphism from $V^k_T$ to $V_T^{2m-k}$.
\end{theorem}

\begin{proof} By Theorem 4.2 in \cite{Wang17}, we know that
$$
A: u\mapsto T\wedge u\wedge\omega^{m-k},
$$
defines an isomorphism from $V^k$ to $V^{2n-k}$. Hence $V^{2n-k}=V_T^{2m-k}$ and the following map
$$
f_T: u\mapsto T\wedge u, \ u\in V^k,
$$
is injective. Thus $f_T$ defines an isomorphism from $V^k$ to  $V^k_T$. Hence $L^{m-k}=A\circ f^{-1}_T$ is an isomorphism  from $V^k_T$ to $V_T^{2m-k}$.
\end{proof}

\begin{definition}\label{de:primitive} We call $T\wedge u\in V^k_T$ a primitive $k$-form if $k\leq m$ and $L^{m-k+1} (T\wedge u)=0$.
\end{definition}

Theorem \ref{th: T-HL} implies:

\begin{theorem}\label{th: T-lef} Every $T\wedge u\in V_T^k$ has an Lefschetz decomposition as follows:
\begin{equation}
T\wedge u=\sum_{r=0}^j L^r  (T\wedge u_r),  \ \qquad \text{for some} \ 0\leq j \leq m,
\end{equation}
where each $T\wedge u_r$ is zero or primitive in $V_T^{k-2r}$. If $T\wedge u=0$ then $T\wedge u_r=0$ for every $r$.
\end{theorem}

\begin{proof} By the isomorphism in Theorem \ref{th: T-HL}, one may assume that $0\leq k\leq m$. Notice that all forms in $V_T^0$ and $V_T^1$ are primitive. Assume that $2\leq k\leq m$, Theorem \ref{th: T-HL} gives $\hat u\in V^{k-2}$ such that
$$
L^{m-k+2} (T\wedge \hat u)= L^{m-k+1}(T\wedge u).
$$
Put $u_0=u-L\hat u$, then  $T\wedge u_0$ is primitive and
$$
T\wedge u=T\wedge u_0+ L(T\wedge \hat u).
$$
Consider $\hat u$ instead $u$,  the Lefschetz decomposition of $T\wedge u$ follows by repeating the above argument. If $T\wedge u=\sum_{r=0}^j  L^r  (T\wedge u_r) =0$ then primitivity of $T\wedge u_r$ for $0\leq r <j$ implies 
$$
0=L^{m-k+j}(\sum_{r=0}^j  L^r  (T\wedge u_r))=L^{m-k+2j} (T\wedge u_j),
$$
which gives $T\wedge u_j=0$ by Theorem \ref{th: T-HL}. By induction on $j$, we get $T\wedge u_r=0$ for every $r$.
\end{proof}

\begin{definition}\label{de:lef-star} If $T\wedge u\in V_T^k$ is primitive then we define
$$
*_s(L_r (T\wedge u)):=(-1)^{[k]} L_{m-r-k} (T\wedge u),
$$
where
$$
 L_p:=\frac{L^p}{p!}, \qquad [k]:=1+\cdots+k=\frac{k(k+1)}{2}.
$$
$*_s$ extends to a $\mathbb{C}$-linear map $*_s: V_T\to V_T$, we call it the Lefschetz star operator on $V_T$. 
\end{definition}

The Lefschetz star operator above is a generalization of the symplectic star operator, see \cite{Wang17} for the background.

\begin{definition}\label{de:triple} Put $\Lambda=*_s^{-1} L*_s$, $B:=[L, \Lambda]$. We call $(L,\Lambda, B)$ the $sl_2$-triple on $V_T$.
\end{definition}

\begin{definition}\label{de:hodge-star} We call $*:=*_s\circ J$ the Hodge star operator on $V_T$, where $J$ is the Weil-operator defined by
$Ju=i^{p-q}u$ if $u\in V^{p,q}_T$.
\end{definition}

Timorin's mixed Hodge-Riemann bilinear relation \cite{Timorin98} gives:

\begin{theorem}\label{th: T-metric} For every non-zero $u\in V^k$, $0\leq k\leq m$, 
$$
\int_{X} u\wedge \overline{ *(T\wedge u)}  > 0,
$$
where $*$ denotes the Hodge star operator on $V_T$.
\end{theorem}

\begin{proof} Let $ T\wedge u=\sum_{r=0}^j L_r  (T\wedge u_r)$
be the Lefschetz decomposition of $T\wedge u$. By our assumption, the degree of $u$ is no bigger than $m$, thus Theorem 4.2 in \cite{Wang17} implies
$$
u=\sum_{r=0}^j L_r u_r.
$$ 
Now primitivity of $T\wedge u_r$ gives
$$
u\wedge \overline{ *(T\wedge u)} =  \sum_{r=0}^j (-1)^{[k-2r]} L_r L_{m+r-k}(T\wedge u_r) \wedge \overline{J(u_r)}. 
$$
By Theorem 4.1 in \cite{Wang17}, if  $u_r$ is not zero then 
$$
(-1)^{[k-2r]} L_r L_{m+r-k}(T\wedge u_r) \wedge \overline{J(u_r)} >0,
$$
as a positive $(n,n)$-form. Thus the theorem follows.
\end{proof}

Let us define
$$
||T\wedge u||^2:=||u||^2_{T, \omega}:=\int_{X} u\wedge \overline{ *(T\wedge u)}, \ u\in V^k, \ 0\leq k\leq m.
$$

\begin{definition}\label{de:T-norm} We call $||T\wedge u||=||u||_{T, \omega}$ the $T$-Hodge theory norm on $V_T^k$.
\end{definition}

\subsection{Proof of Lemma \ref{le:2-new}}  \eqref{eq:G-new} follows directly from the definition of the $T$-Hodge theory norm. For \eqref{eq:Gt-new}, notice that
$$
\frac{d}{dt}\left (\frac{\omega^m}{m!} \wedge T\right)=\theta\wedge \frac{\omega^{m-1}}{(m-1)!} \wedge T,
$$
gives 
\begin{equation}\label{eq:theta1-p}
(\theta+G \frac{\omega}{m}) \wedge \frac{\omega^{m-1}}{(m-1)!} \wedge T =0.
\end{equation}

\begin{definition}\label{de:theta1} $\theta_0:=\theta+G \frac{\omega}{m}, \ \theta_1:= -\frac Gm, \ \theta':=-\theta_0 \wedge \frac{\omega^{m-2}}{(m-2)!}+\theta_1 \wedge   \frac{\omega^{m-1}}{(m-1)!}$.
\end{definition}

We have $\theta=\theta_0+\theta_1\omega$. \eqref{eq:theta1-p} implies that $T\wedge \theta_0$ is primitive. Thus we have
\begin{equation}
T\wedge \theta'= *(T\wedge \theta)=\overline{*(T\wedge \theta)}.
\end{equation}
Apply the derivative of \eqref{eq:theta1-p} with respect to $t$, we get
$$
(G_t \frac{\omega}{m} + G \frac{\theta}{m} ) \wedge \frac{\omega^{m-1}}{(m-1)!} \wedge T + \theta_0\wedge \theta \wedge \frac{\omega^{m-2}}{(m-2)!} \wedge T=0,
$$
thus
\begin{eqnarray*}
G_t \, \frac{\omega^{m}}{m!} \wedge T & = & \theta_1\theta\wedge \frac{\omega^{m-1}}{(m-1)!} \wedge T- \theta_0\wedge \theta \wedge \frac{\omega^{m-2}}{(m-2)!} \wedge T \\ 
 & = & \theta\wedge\theta'\wedge T=  \theta\wedge \overline{*(T\wedge \theta)},
\end{eqnarray*}
which gives \eqref{eq:Gt-new}. Now it suffices to prove \eqref{eq:Gtheta-new}. Notice that Definition \ref{de:triple} gives
$$
\Lambda(T\wedge\theta)= *_s^{-1} (\omega \wedge T\wedge\theta')=T\wedge m\theta_1 =-T\wedge G.
$$
Thus \eqref{eq:Gtheta-new} is true.

\section{H\"ormander $L^2$-estimate in $T$-Hodge theory}

\textbf{Notation}: In this paper, $d^*$ and $(d^c)^*$ denote the adjoint of $d$ and $d^c$ with respect to the $T$-Hodge theory norm. 

\begin{theorem}\label{th:Hormander-new} Let $(X, \hat\omega)$ be an $n$-dimensional complete K\"ahler manifold. Let 
$$
T:=\alpha_{m+1}\wedge \cdots \wedge \alpha_n, \ 2\leq m\leq n,
$$
be a finite wedge product of K\"ahler forms on $X$ such that \eqref{eq:equ-norm} is true. Let $\theta$ be a smooth $d$-closed $2$-form on $X$. Assume that the $T$-Hodge theory norm $||T\wedge\theta||$ is finite. Then there exists a smooth solution of $$d(T\wedge u)= (d^c)^*(T\wedge\theta)$$ such that $||T\wedge u|| \leq ||T\wedge \theta||$.  
\end{theorem}

\begin{proof}  The proof contains two steps.

\medskip

\emph{Step 1}: "A prior estimate"
\begin{equation}\label{eq:claim1}
|(T\wedge \alpha, (d^c)^* (T\wedge \theta))|^2\leq ||T\wedge \theta||^2Q(\alpha, \alpha), 
\end{equation}
for every smooth $1$-form $\alpha$ with compact support in $X$, where
$$
Q(\alpha, \alpha):=||d(T\wedge \alpha)||^2+ ||d^*(T\wedge \alpha )||^2.
$$

\medskip

\emph{Proof of Step 1}: Since $$(T\wedge\alpha , (d^c)^* (T\wedge \theta))=(d^c(T\wedge\alpha ), T\wedge\theta),$$ it suffices to show the following $T$-geometry version of the Bochner-Kodaira-Nakano identity
$$
||d(T\wedge \alpha )||^2+ ||d^*(T\wedge\alpha)||^2=||d^c(T\wedge\alpha)||^2+ ||(d^c)^*(T\wedge\alpha)||^2,
$$
which is a special case of Theorem 4.8 in \cite{Wang17}.

\medskip 

\emph{Step 2}: By \emph{Step 1}, we know that 
$$
F: \alpha \mapsto (T\wedge\alpha, (d^c)^* (T\wedge\theta)),
$$
is $Q$-bounded by $||T\wedge\theta||$. Thus $F$ extends to  a bounded linear functional on the $Q$-completion, say $H$, of the space of smooth $1$-forms with compact support in $X$. The Riesz representation theorem gives $\beta\in H$ with
\begin{equation}\label{eq:es}
Q(\beta, \beta) \leq ||T\wedge\theta||^2,
\end{equation}
such that
\begin{equation}\label{eq:esq1}
Q(\alpha, \beta)=F(\alpha)=(T\wedge\alpha, (d^c)^* (T\wedge\theta)), 
\end{equation}
for every smooth $1$-form $\alpha$ with compact support in $X$, where
\begin{equation}\label{eq:esq2}
Q(\alpha, \beta)=(d(T\wedge\alpha), d(T\wedge\beta ))+ (d^*(T\wedge \alpha), d^*(T\wedge\beta)).
\end{equation}
Since $H$ is a subspace of the space of currents, we have
\begin{equation}\label{eq:esq3}
Q(\alpha, \beta)=(T\wedge\alpha , (dd^*+d^*d)(T\wedge\beta )).
\end{equation}
Thus \eqref{eq:esq1} and \eqref{eq:esq3} together give
$$
(dd^*+d^*d)(T\wedge \beta )=(d^c)^* (T\wedge\theta),
$$
in the sense of current. Let us define $u$ such that $T\wedge u= d^*(T\wedge\beta)$. Since $dd^*+d^* d$ is elliptic, we know that $\beta$ is smooth. Thus $u$ is smooth. Notice that \eqref{eq:es} gives
$$||T\wedge u|| \leq ||T\wedge\theta||,$$ Thus   it suffices to prove the following identity.
\end{proof}

\begin{lemma}\label{le:last-1} $d^*d(T\wedge \beta ) \equiv 0$.
\end{lemma}

\begin{proof} The $T$-K\"ahler identity $(d^c)^*=[d,\Lambda]$ (see section 4 in \cite{Wang17}) implies that $$d(d^c)^*+(d^c)^*d=0.$$ Thus
$$
d(d^c)^*(T\wedge \theta)=-(d^c)^*d(T\wedge \theta)=0.
$$
Now we have
$$
dd^* d(T\wedge \beta) \equiv 0.
$$
Since $\hat \omega$ is complete, there exists a smooth exhaustion function, say $\rho$, on $X$ such that
\begin{equation}\label{eq:db}
|d\rho|_{\hat \omega} \leq 1. 
\end{equation}
Let  $0\leq \chi \leq 1$ be a smooth function on $\mathbb R$ such that $\chi\equiv 1$ on $(-\infty, 1)$ and $\chi\equiv 0$ on $(2, \infty)$. Then for each $\varepsilon>0$, $\chi(\varepsilon \rho)$ is a smooth function with compact support. Since
\begin{equation}\label{eq:chen}
(\chi^2(\varepsilon b ) dd^*d (T\wedge\beta),  d(T\wedge\beta))=0,
\end{equation}
and
$$
\chi^2(\varepsilon b ) dd^*d (T\wedge\beta)= d(\chi^2(\varepsilon b )d^*d (T\wedge\beta))- 2 d(\chi(\varepsilon b))\wedge \chi (\varepsilon b)d^*d(T\wedge\beta),
$$
we have
\begin{equation}\label{eq:equi}
||\chi(\varepsilon b) d^*d(T\wedge \beta)||^2=2 (d(\chi(\varepsilon b))\wedge \chi (\varepsilon b)d^*d(T\wedge \beta), d(T\wedge\beta)). 
\end{equation}
Thus Lemma \ref{le:last-1} follows from the following estimate
\begin{equation}\label{eq:estimate}
\lim_{\varepsilon\to 0} ||d(\chi(\varepsilon b))\wedge \chi (\varepsilon b)d^*d(T\wedge \beta)||=0.
\end{equation} 
The above estimate is easily seen to be true in case $T=1$, see \cite{CWW}. The general case will be proved in the appendix.
\end{proof}

\subsection{Proof of Lemma \ref{le:3-new}}

By Lemma \ref{le:2-new}, we have
$$
d\left(  T\wedge(E_\mu(G)-G) \right)=d\Lambda(T\wedge \theta)=[d,\Lambda] (T\wedge\theta),
$$
By the K\"ahler identity in $T$-Hodge theory (section 4 in \cite{Wang17}), we have $[d,\Lambda]=(d^c)^*$, thus  $T\wedge(E_\mu(G)-G)$ is a solution of 
$$
d(\cdot)=(d^c)^*(T\wedge \theta).
$$ 
Notice that $T\wedge(E_\mu(G)-G)$ is perpendicular to $\ker d$, thus it is also the $L^2$-minimal solution. By \eqref{eq:equ-norm}, for every fixed $0<t<1$, $\omega=t\alpha_1+(1-t)\alpha_2$ is complete. Apply Theorem \ref{th:Hormander-new} to the case $\hat\omega=\omega$, Lemma \ref{le:3-new} follows.

\section{Proof of the Alexandrov-Fenchel inequality}

\begin{lemma}\label{le:complete} Put
$$
\psi(x)=\sum_{j=1}^n \log\frac{1}{1+(x^j)^2}  +C \log(1+e^{x^j}), \ C:= 4(1+e^{\sqrt{3}})^2 e^{\sqrt{3}}.
$$
Then $\psi$ is strictly convex on $\mathbb R^n$ and  $\nabla \psi(\mathbb R^n) \subset (-1,C+1)^n$. Moreover, if we look at $\psi$ as a function on $\mathbb R^n \times {\mathbb T}^n$ then $dd^c\psi$ is complete K\"ahler on $\mathbb R^n \times {\mathbb T}^n$. 
\end{lemma}

\begin{proof} A direct computation gives
\begin{equation}\label{eq: c-1}
\left(\log\frac{1}{1+(x^j)^2} \right)_{x^j} =\frac{-2x^j}{1+(x^j)^2}, 
\end{equation}
and
\begin{equation}\label{eq: c-2}
\left(\log\frac{1}{1+(x^j)^2} \right)_{x^jx^j} =\frac{2(x^j)^2-2}{(1+(x^j)^2)^2} \geq \frac{1}{1+(x^j)^2}, \ \ \text{if} \ (x^j)^2\geq 3.
\end{equation}
Since $\log (1+e^x)$ is convex, the above inequality gives
$$
\psi_{x^jx^j} \geq  \frac{1}{1+(x^j)^2} \ \ \text{if} \ (x^j)^2 \geq  3. 
$$
We also have
\begin{equation}\label{eq: c-3}
\left(\log(1+e^{x^j}) \right)_{x^jx^j} =\frac{e^{x^j}}{(1+e^{x^j})^2} \geq \frac{e^{-\sqrt{3}}}{(1+e^{\sqrt{3}})^2}, \ \ \text{if} \ (x^j)^2\leq 3.
\end{equation}
Thus
\begin{equation}\label{eq: c-4}
C \left(\log(1+e^{x^j}) \right)_{x^jx^j} \geq 4 \geq \frac{4}{1+(x^j)^2}, \ \ \text{if} \ (x^j)^2\leq 3, 
\end{equation}
which gives
$$
\psi_{x^jx^j} \geq  \frac{4}{1+(x^j)^2} + \frac{2(x^j)^2-2}{(1+(x^j)^2)^2} \geq \frac{2}{1+(x^j)^2} \ \ \text{if} \ (x^j)^2\leq 3. 
$$
Notice that $\psi_{x^jx^k}=0$ if $j\neq k$. Thus $\psi$ is strictly convex and 
$$
dd^c\psi \geq \sum_{j=1}^n \frac{1}{1+(x^j)^2}  dx^j \wedge dy^j, 
$$
on $\mathbb R^n \times {\mathbb T}^n$. Denote by $g$ the associated Riemannian metric of $dd^c\psi$, then we have
$$
g\geq g_0:=\sum_{j=1}^n  \frac{1}{1+(x^j)^2}  (dx^j 
\otimes dx^j +dy^j\otimes dy^j).
$$
Thus
$$
|dx^j|_g \leq |dx^j|_{g_0}=\sqrt{1+(x^j)^2}.
$$
Since $d\log (1+|x|^2)=\sum_{j=1}^n \frac{2x^j dx^j}{1+|x|^2}$, we have
$$
|d\log (1+|x|^2)|_{g} \leq \sum_{j=1}^n \frac{2|x^j|}{1+|x|^2}|dx^j|_g \leq \sum_{j=1}^n \frac{2|x^j|}{1+|x|^2} \sqrt{1+(x^j)^2} \leq n. 
$$
Notice that $\log(1+|x|^2)$ is an exhaustion function on $\mathbb R^n \times {\mathbb T}^n$, the above inequality implies that $dd^c\psi$ is complete K\"ahler. $\nabla \psi(\mathbb R^n) \subset (-1, C+1)^n$ follows from
$$
\psi_{x^j}= \frac{-2x^j}{1+(x^j)^2} + C\frac{e^{x^j}}{1+e^{x^j}},  \ 2|x_j| \leq 1+(x^j)^2, \ 0 < \frac{e^{x^j}}{1+e^{x^j}} <1. 
$$
The proof is complete. 
\end{proof}

We shall use our main theorem and the above lemma to prove Theorem \ref{th:AF-m-c}, which implies the Alexandrov-Fenchel inequality.

\subsection{Proof of Theorem \ref{th:AF-m-c}} 

Put
$$
\tilde \phi= \psi+\phi_1+\phi_2+\phi_{m+1}+\cdots+\phi_n.
$$
The above lemma implies that $\hat \omega:=dd^c\tilde \phi$
is complete on $\mathbb R^n \times {\mathbb T}^n$ and $dd^c
\phi_j \leq \hat \omega$ for each $j$. Moreover, by the above lemma, $\nabla\psi(\mathbb R^n)$ is bounded, thus $\nabla\tilde\phi(\mathbb R^n)$ is bounded and $(X,\hat\omega)$ has finite volume. We know that Theorem \ref{th:AF-m-c} follows from Theorem \ref{th:main}.

\section{Appendix}

\subsection{Compare the $T$-Hodge theory norm with the usual norm}

For every smooth $k$-form $u$, $0\leq k\leq m$, on $X$, let us define $|u|^2_{T,\omega}$ such that
$$
u\wedge \overline{ *(T\wedge u)}=|u|^2_{T,\omega} \frac{\omega^m}{m!} \wedge T.
$$
where $*$ denotes the Hodge star operator on $V_T$, see Definition \ref{de:hodge-star}.

\begin{definition} We call $|u|_{T,\omega}$ the pointwise $T$-norm of $u$.
\end{definition}

\begin{lemma}\label{le:c1} Let $|T\wedge u|_\omega$ be the usual pointwise norm of $T\wedge u$ with respect to $\omega$. If $T=\omega^{n-m}$ then 
$$
\frac{n!(n-m)!)}{m!} |u|^2_{T,\omega} \leq |T\wedge u|^2_{\omega} \leq \frac{(n!)^2}{(m!)^2}  |u|^2_{T,\omega}.
$$
\end{lemma}

\begin{proof} By Definition \ref{de:primitive}, if $T=\omega^{n-m}$ then a form $T\wedge v\in V_T^k$ is primitive in $T$-Hodge theory if and only if $v$ is primitive with respect to $\omega$ in the usual sense. Let
$$
T\wedge u:=\sum_{r=0}^j L_r  (T\wedge u_r)=\sum_{r=0}^j L_{n-m+r} u'_r, \ u'_r:=\frac{(n-m+r)!}{r!} u_r,
$$
be the Lefschetz decomposition of $T\wedge u$. Then Definition \ref{de:hodge-star} gives
$$
*(T\wedge u)=\sum_{r=0}^j (-1)^{[k-2r]} L_{m-k+r} (T\wedge Ju_r).
$$
Moreover,
$$
\star (T\wedge u)=\sum_{r=0}^j  (-1)^{[k-2r]}L_{m-k+r} (Ju'_r),
$$
where $\star$ denotes the usual Hodge star operator. Recall that 
$$
T\wedge u \wedge \overline{\star (T\wedge u)}=|T\wedge u|^2_{\omega}\frac{\omega^n}{n!}.
$$
Thus the lemma follows.
\end{proof}

For general $T=\alpha_{m+1}\wedge \cdots \wedge \alpha_n$, we have:

\begin{lemma}\label{le:c2} Assume that \eqref{eq:equ-norm} is true. Then there exists a constant $C_1$ that only depends on $C, n,m$ such that
$$
C_1^{-1} |u|_{T,\hat \omega} \leq |T\wedge u|_{\hat \omega} \leq C_1  |u|_{T,\hat \omega}.
$$
\end{lemma}

\begin{proof} By Lemma \ref{le:c1}, it suffices to compare $|u|^2_{T,\hat \omega}$ with $|u|^2_{T_0,\hat \omega}$, where $T_0:={\hat \omega}^{n-m}$. Fix an arbitrary point, say $z_0$, in $X$, let us choose local coordinates, say $\{z^j\}$, near $z_0$ such that
$$
 \hat \omega(z_0)=i \sum_{j=1}^n  dz^j \wedge d\bar z^j. 
$$
With respect to the local coordinates $\{z^j\}$, we can identify the space of positive $(1,1)$-forms at $z_0$ with the space of positive definite $n$ by $n$ Hermitian matrices. We know that every positive definite $n$ by $n$ Hermitian matrix can be written as
$$
A=O B O^*, \ OO^*=I_n,
$$
where $O^*$ denotes the conjugate transpose of $O$, $I_n$ is the identity matrix and $B$ is a diagonal matrix with positive eigenvalues. Moreover, 
$$
\frac{\omega(z_0)}{C} \leq \hat \omega(z_0) \leq C \omega(z_0)
$$
if and only if each eigenvalue of the associated matrix of  $ \omega(z_0)$ lies in $[1/C, C]$. Consider 
$$
V:= U(n)\times [1/C, C]^n,
$$
where $U(n):=\{O: OO^*=I_n\}$ is the unitary group.  Every element, say $v=(O, \lambda_1, \cdots, \lambda_n)$, in $V$ represents a positive $(1,1)$-form, say $\omega^v$, at $z_0$ whose associated matrix is  
$$
O{\rm Diag}\{\lambda_1, \cdots, \lambda_n\} O^*.
$$ 
Consider the following map, say $F$,  from 
$$
V^{n-m}:=\underbrace{V\times\cdots \times V}_{n-m}
$$ 
to the space of Hermitian norms on $\wedge^{k} (\mathbb C\otimes T_{z_0}^*X)$, $0\leq k\leq m$, defined by
$$
(v^{m+1}, \cdots, v^n)\mapsto |\cdot|_{T, \hat \omega(z_0)}, \ T:= \omega^{v^{m+1}} \wedge \cdots \wedge \omega^{v^n}.
$$
The lemma follows since $V^{n-m}$ is compact and connected.
\end{proof}

\subsection{Proof of estimate \eqref{eq:estimate}}
Let us write $d^*d(T\wedge \beta)$ as $T\wedge \sigma$, where $\sigma$ is a one-form. Then
$$
 ||d(\chi(\varepsilon \rho))\wedge \chi (\varepsilon \rho)d^*d(T\wedge\beta)||^2=\int_{X} |d(\chi(\varepsilon \rho))\wedge \chi (\varepsilon \rho)\sigma|^2_{T, \hat \omega} \frac{\hat \omega^m}{m!} \wedge T.  
$$
By Lemma \ref{le:c2}, we have
$$
|d(\chi(\varepsilon \rho))\wedge \chi (\varepsilon \rho)\sigma|_{T, \hat \omega} \leq C_1 |d(\chi(\varepsilon \rho))\wedge \chi (\varepsilon \rho)d^*d(T\wedge\beta)|_{\hat \omega}.
$$
Since $|d\rho|_{\hat \omega} \leq 1$, we have
$$
|d(\chi(\varepsilon \rho))\wedge \chi (\varepsilon \rho)d^*d(\beta\wedge T)|_{\hat \omega} \leq \left(\varepsilon \sup|\chi'|\right)
|\chi (\varepsilon \rho)d^*d(T\wedge\beta)|_{\hat \omega}.
$$
Use Lemma \ref{le:c2} again, we get
$$
|d(\chi(\varepsilon \rho))\wedge \chi (\varepsilon \rho)\sigma|_{T, \hat \omega} \leq  \left(\varepsilon C_1^2 \sup|\chi'|\right)  |\chi (\varepsilon \rho)\sigma|_{T, \hat \omega}, 
$$
which gives
$$
 ||d(\chi(\varepsilon \rho))\wedge \chi (\varepsilon \rho)d^*d(T\wedge\beta)|| \leq \left( \varepsilon C_1^2 \sup|\chi'| \right)  || \chi (\varepsilon \rho)d^*d(T\wedge\beta)|| .
$$ 
By \eqref{eq:equi}, then we have
$$
 || \chi (\varepsilon \rho)d^*d(T\wedge\beta)||^2 \leq 2  \left( \varepsilon C_1^2 \sup|\chi'| \right)  || \chi (\varepsilon \rho)d^*d(T\wedge\beta)|| \cdot  ||T\wedge \theta ||,
$$
hence
$$
 || \chi (\varepsilon \rho)d^*d(T\wedge\beta)||\leq  \left( 2  \varepsilon C_1^2 \sup|\chi'| \right) ||T\wedge \theta||,
$$
which gives
$$
||d(\chi(\varepsilon \rho))\wedge \chi (\varepsilon \rho)d^*d(T\wedge \beta)|| \leq 2 ( \varepsilon C_1^2 \sup|\chi'|)^2 ||T\wedge \theta||,  
$$
thus \eqref{eq:estimate} follows.


\begin{thebibliography}{99}

\bibitem{BBN} K. Ball, F. Barthe and A. Naor, {\it Entropy jumps in the presence of a spectral gap}, Duke Math. J. {\bf 119}, 41--63.

\bibitem{B-B} R. J. Berman and B. Berndtsson, {\it Real Monge-Amp\`{e}re equations and K\"ahler-Ricci solitons on toric log Fano varieties}, Annales de la facult\'e des sciences de Toulouse Math\'ematiques, {\bf 22} (2013), 649--711.

\bibitem{B98} B. Berndtsson, {\it Pr\'ekopa's theorem and Kiselman's minimum principle for plurisubharmonic functions}, Math. Ann. {\bf 312} (1998), 785--792.

\bibitem{Bo-notes} B. Berndtsson, {\it Notes on complex and convex geometry}, in www.math.chalmers.se/$\sim$bob/3notes.pdf 

\bibitem{Bern06} B. Berndtsson, {\it Subharmonicity properties of the Bergman kernel and some other functions associated to pseudoconvex domains}, Ann. Inst. Fourier (Grenoble), {\bf 56} (2006), 1633--1662.

\bibitem{Bern09} B. Berndtsson, {\it Curvature of vector bundles associated to holomorphic fibrations}, Ann. Math. {\bf 169} (2009), 531--560.

\bibitem{Bern-notes} B. Berndtsson, {\it Convexity on the space of K\"ahler metrics}. Ann. Fac. Sci. Toulouse Math. {\bf 22} (2013), 713--746.

\bibitem{Bern-CBM} B. Berndtsson, {\it Real and complex Brunn-Minkowski theory}. Analysis and geometry in several complex variables, 1--27, Contemp. Math., 681, Amer. Math. Soc., Providence, RI, 2017.

\bibitem{BPW} B. Berndtsson, M. P\u aun and X. Wang, {\it Algebraic fiber spaces and curvature of higher direct images},  arXiv:1704.02279.

\bibitem{BS} B. Berndtsson and N. Sibony, {\it The $\dbar$-equation on a positive current}, Invent. Math. {\bf 147} (2002), 371--428.

\bibitem{BL76} H. J. Brascamp and E. H. Lieb, {\it On extensions of the Brunn-Minkowski and Pr\'ekopa-Leindler theorems, including inequalities for log concave functions, and with an application to the diffusion equation}, Journal of Functional Analysis, {\bf 22}  (1976), 366--389.

\bibitem{BZ} Y. D. Burago and V. A. Zalgaller, {\it Geometric inequalities}, translated from the Russian by A. B. Sosinski\u{i}. Grundlehren der Mathematischen Wissenschaften, 285. Springer Series in Soviet Mathematics (1988).

\bibitem{Cattani08} E. Cattani, {\it Mixed Lefschetz theorems and Hodge-Riemann bilinear relations}, International Mathematics Research, Vol. 2008, no. 10, Article ID rnn025, 20 pages.


\bibitem{CWW} B. Y. Chen, J. J. Wu and X. Wang, {\it Ohsawa-Takegoshi type theorem and extension of plurisubharmonic functions}. Math. Ann. {\bf 362} (2015), 305--319.

\bibitem{Co} D. Cordero-Erausquin, {\it On Berndtsson's generalization of Pr{\'{e}}kopa's theorem}, Math. Z. {\bf 249} (2005), 401--410.

\bibitem{CK} D. Cordero-Erausquin and B. Klartag, {\it Moment measures}, Journal of Functional Analysis, {\bf  268} (2015), 3834--3866.

\bibitem{Demailly82} J.-P. Demailly, {\it Estimations $L^2$ pour l'op\'{e}rateur $\bar{\partial}$ d'un
    fibr\'{e} vectoriel holomorphe semi-positif au-dessus d'une vari\'{e}t\'{e} k\"{a}hl\'{e}rienne
    compl\`{e}te}, Ann. Sci. \'{E}cole Norm. Sup. {\bf 15} (1982), 457--511.


\bibitem{DN06} T. C. Dinh, V. A. Nguy\^en, {\it The mixed Hodge-Riemann bilinear relations for compact K\"ahler manifolds}, Geometric and Functional Analysis {\bf 16} (2006), 838--849.

\bibitem{Graham} W. Graham, {\it Logarithmic convexity of push-forward measures}, Invent. math. {\bf 123} (1996), 315--322.

\bibitem{Gardner} R. Gardner, {\it The Brunn-Minkowski inequality}, Bulletin of the American Mathematical Society, {\bf 39} (2002), 355--405.

\bibitem{Gromov} M. Gromov, {\it Convex sets and K\"ahler manifolds}, Advances in differential geometry and topology, (1990), 1--38.

\bibitem{GMTZ}  P. Guan, X. N. Ma, N. Trudinger and X. Zhu, {\it A form of Alexandrov-Fenchel inequality}, Pure and Applied Mathematics Quarterly, {\bf 6} (2010),  999--1012.

\bibitem{Hormander65} L. H\"{o}rmander, {\it $L^2$-estimates and existence theorems for the $\op$-operator}, Acta Math. {\bf 113} (1965), 89--152.

\bibitem{Hormander66} L. H\"{o}rmander, {\it An introduction to complex analysis in several variables}, Van Nostrand, Princeton, 1966.

\bibitem{Khovanskii88} A. G. Khovanski\u{i}, {\it Algebra and mixed volumes},  In {\it Geometric Inequalities}, edited by D. Yu. Burago and V. A. Zalgaller. Grundlehren der Mathematischen Wissenschaften 285. Berlin: Springer, 1988.

\bibitem{KK12} K. Kaveh and A. G. Khovanski\u{i}, {\it Newton-Okounkov bodies, semigroups of integral points, graded algebras and intersection theory}, Ann. of Math. {\bf 176} (2012), 925--978.

\bibitem{KK12-1} K. Kaveh and A. G. Khovanski\u{i}, {\it Algebraic equations and convex bodies}, Perspectives in analysis, geometry, and topology. Birkh\"auser Boston, 2012, 263--282.

\bibitem{Ly} L. Lusternik, {\it Die Brunn-Minkowskische Ungleichung f\"ur beliebige messbare Mengen}, C. R. Acad. Sci. URSS {\bf 8} (1935), 55--58. 

\bibitem{MY04} F. Maitani, H. Yamaguchi, {\it Variation of Bergman metrics on Riemann surfaces}, Math. Ann. {\bf 330} (2004), 477--489.

\bibitem{MR13} V. Milman and L. Rotem, {\it Mixed integrals and related inequalities}, Journal of Functional Analysis, {\bf 264} (2013), 570--604.


\bibitem{Prekopa73} A. Pr\'ekopa, {\it On logarithmic concave measures and functions}, Acad. Sci. Math. (Szeged) {\bf 34} (1973), 335--343.

\bibitem{Schneider} R. Schneider, {\it Convex bodies: the Brunn-Minkowski theory}, Cambridge University press, 2013.

\bibitem{Timorin98} V. A. Timorin, {\it Mixed Hodge-Riemann bilinear relations in a linear context}, Funktsional.
Anal i Prilozhen {\bf 32} (1998), 63--68, 96.

\bibitem{Wang-k} X. Wang, {\it A flat Higgs bundle structure on the complexified K\"ahler cone}, arXiv:1612.02182

\bibitem{Wang17} X. Wang, {\it Notes on variation of Lefschetz star operator and T-Hodge theory}, arXiv:1708.07332 


\bibitem{Nystrom15} D. Witt Nystr\"om, {\it Canonical growth conditions associated to ample line bundles},  arXiv:1510.00510 


\end{thebibliography}
\end{document}